\documentclass{amsproc}
\usepackage{amssymb}
\usepackage{amsfonts}

\theoremstyle{plain}

\newtheorem{corollary}{Corollary}

\newtheorem{definition}{Definition}

\newtheorem{proposition}{Proposition}

\newtheorem{theorem}{Theorem}
\numberwithin{equation}{section}
\input{tcilatex}

\begin{document}
\title[On The Grand Wiener Amalgam Space]{On The Grand Wiener Amalgam Spaces}
\author{A.Turan G\"{u}rkanl\i }
\address{Istanbul Arel University Faculty of Science and Letters Department
of Mathematics and Computer Sciences}
\email{turangurkanli@arel.edu.tr}
\date{}

\begin{abstract}
In this article, notations are included in Section 1. In Section 2, we
define the grand Wiener amalgam space by using the classical Wiener amalgam
space $\left[ 9,15,16,17\right] $ and the generalized grand Lebesgue space $%
\left[ 18,13\right] .$ Section 3, concerns the inclusions between these
spaces and some applications. In last section Section 4, we prove the H\"{o}%
lder's inequality for grand Wiener amalgam space. We also find the associate
space and dual of this space, and we prove that the grand Wiener amalgam
space is not reflexive.
\end{abstract}

\maketitle

\section{\textbf{Notations}}

Let $\Omega $ be locally compact hausdorff space and let $\left( \Omega ,%
\mathfrak{B},\mu \right) $ be finite Borel measure space. The grand Lebesgue
space $L^{p)}\left( \mu \right) $ was introduced in $\left[ 18\right] $ by
the norm

\begin{equation*}
\left\Vert {f}\right\Vert _{p)}=\sup_{0<\varepsilon \leq p-1}\left(
\varepsilon \int_{\Omega }\left\vert f\right\vert ^{p-\varepsilon }d\mu
\right) ^{\frac{1}{p-\varepsilon }}
\end{equation*}%
where $1<p<\infty .$ This is a Banach space. For $0<\varepsilon \leq p-1,$ $%
L^{p}\left( \mu \right) \subset L^{p)}\left( \mu \right) \subset $ $%
L^{p-\varepsilon }\left( \mu \right) $ hold. For some properties and
applications of $L^{p)}\left( \mu \right) $ spaces we refer to papers $\left[
1\right] ,\left[ 4\right] ,\left[ 6\right] ,\left[ 7\right] ,$ $\left[ 12%
\right] $and $\left[ 13\right] .$ A generalization of the grand Lebesgue
spaces are the spaces $L^{p),\theta }\left( \mu \right) ,$ $\theta \geq 0,$
defined by the norm $\left( \text{see }\left[ 1\right] ,\left[ 13\right]
\right) $

\begin{equation*}
\left\Vert f\right\Vert _{p),\theta ,\mu }=\left\Vert f\right\Vert
_{p),\theta }=\sup_{0<\varepsilon \leq p-1}\varepsilon ^{\frac{\theta }{%
p-\varepsilon }}\left( \int_{\Omega }\left\vert f\right\vert ^{p-\varepsilon
}d\mu \right) ^{\frac{1}{p-\varepsilon }}=\sup_{0<\varepsilon \leq
p-1}\varepsilon ^{\frac{\theta }{p-\varepsilon }}\left\Vert {f}\right\Vert
_{p-\varepsilon };
\end{equation*}%
when $\theta =0$ the space $L^{p),0}\left( \mu \right) $ reduces to Lebesgue
space $L^{p}\left( \mu \right) $ and when $\theta =1$ the space $%
L^{p),1}\left( \mu \right) $ reduces to grand Lebesgue space $L^{p)}\left(
\mu \right) .$ We have for all $1<p<\infty $ and $\varepsilon >0$ 
\begin{equation*}
L^{p}\left( \mu \right) \subset L^{p),\theta }\left( \mu \right) \subset
L^{p-\varepsilon }\left( \mu \right) .
\end{equation*}%
Different properties and applications of these spaces were discussed in $%
\left[ 1\right] ,\left[ 13\right] $ and $\left[ 4\right] .$

Let\ $1\leq p<\infty ,\theta \geq 0,$ and $J$ be the one of the set of $%
%TCIMACRO{\U{2115} }%
%BeginExpansion
\mathbb{N}
%EndExpansion
,$ $%
%TCIMACRO{\U{2124} }%
%BeginExpansion
\mathbb{Z}
%EndExpansion
$ or $%
%TCIMACRO{\U{2124} }%
%BeginExpansion
\mathbb{Z}
%EndExpansion
^{n}.$ We define the grand Lebesgue sequence space $\ell ^{p),\theta }=\ell
^{p),\theta }\left( J\right) $ by the norm%
\begin{equation*}
\left\Vert u\right\Vert _{\ell ^{p),\theta }\left( J\right)
}=\sup_{0<\varepsilon \leq p-1}\left( \varepsilon ^{\theta
}\dsum\limits_{k\in J}\left\vert u_{k}\right\vert ^{p-\varepsilon }\right) ^{%
\frac{1}{p-\varepsilon }}=\sup_{0<\varepsilon \leq p-1}\varepsilon ^{\frac{%
\theta }{p-\varepsilon }}\left\Vert u\right\Vert _{\ell ^{p-\varepsilon
}\left( J\right) }.
\end{equation*}

Let $p^{^{\prime }}=\frac{p}{p-1},$ $1<p<\infty .$ First consider an
auxiliary space namely $L^{(p^{\prime },\theta }\left( \Omega \right)
,\theta >0,$ defined by 
\begin{equation*}
\left\Vert g\right\Vert _{(p^{^{\prime }},\theta
}=\inf_{g=\dsum\limits_{k=1}^{\infty }g_{k}}\left\{
\dsum\limits_{k=1}^{\infty }\inf_{0<\varepsilon \leq p-1}\varepsilon ^{-%
\frac{\theta }{p-\varepsilon }}\left( \int_{\Omega }\left\vert
g_{k}\right\vert ^{\left( p-\varepsilon \right) ^{^{\prime }}}dx\right) ^{%
\frac{1}{\left( p-\varepsilon \right) ^{^{\prime }}}}\right\}
\end{equation*}%
where the functions $g_{k},k\in 
%TCIMACRO{\U{2115} }%
%BeginExpansion
\mathbb{N}
%EndExpansion
,$ being in $\mathcal{M}_{0},$ the set of all real valued measurable
functions, finite a.e. in $\Omega .$ After this definition the generalized
small Lebesgue spaces have been defined by 
\begin{equation*}
L^{p)^{^{\prime }},\theta }\left( \Omega \right) =\left\{ g\in \mathcal{M}%
_{0}:\left\Vert g\right\Vert _{p)^{^{\prime }},\theta }<+\infty \right\} ,
\end{equation*}%
where%
\begin{equation*}
\left\Vert g\right\Vert _{p)^{^{\prime },\theta }}=\sup_{\substack{ 0\leq
\psi \leq \left\vert g\right\vert  \\ \psi \in L^{(p^{\prime },\theta }}}%
\left\Vert \psi \right\Vert _{(p^{^{\prime }},\theta }.
\end{equation*}%
For $\theta =0$ it is $\left\Vert f\right\Vert _{(p^{^{\prime
}},0}=\left\Vert f\right\Vert _{p^{^{\prime }},\theta },\left[ 5\right] ,%
\left[ 11\right] .$

Let $1\leq p,q\leq \infty .$ The space $\left( L^{p)}\right) _{loc}$consists
of (classes of) measurable functions $f$ $:\Omega \rightarrow \mathbb{C}$
such that $f\chi _{K}\in L^{p)},$ for any compact subset $K\subset \Omega ,$%
where $\chi _{K}$ is the characteristic function of $f.$ It is a topological
vector space with the family of seminorms $f\rightarrow \left\Vert
f\right\Vert _{p)}.$ Since $L^{p}\subset L^{p)},$ it is easy to show that $%
\left( L^{p}\right) _{loc}\hookrightarrow \left( L^{p)}\right) _{loc}.$

\section{The \textbf{Grand Wiener Amalgam Space}}

\begin{definition}
Let $\Omega $ be locally compact hausdorff space and let $\left( \Omega ,%
\mathfrak{B},\mu \right) $ be finite Borel measure space. Also assume that $%
1<p,q<\infty $ and $Q\subset \Omega $ is a fix compact subset with nonemty
interior. The grand Wiener amalgam space $W\left( L^{p),\theta
_{1}},L^{q),\theta _{2}}\right) $ consists of all functions (classes of) $%
f\in \left( L^{p),\theta _{1}}\right) _{loc}$ such that the control function 
\begin{equation*}
F_{f}^{p),\theta _{1}}\left( x\right) =F_{f}^{p),\theta _{1},Q}\left(
x\right) =\left\Vert f.\chi _{Q+x}\right\Vert _{p),\theta
_{1}}=\sup_{0<\varepsilon \leq p-1}\varepsilon ^{\frac{\theta _{1}}{%
p-\varepsilon }}\left\Vert f.\chi _{Q+x}\right\Vert _{p-\varepsilon }
\end{equation*}%
lies in $L^{q),\theta _{2}}.$ The norm of $W\left( L^{p),\theta
_{1}},L^{q),\theta _{2}}\right) $ defined by 
\begin{equation*}
\left\Vert f\right\Vert _{W\left( L^{p),\theta _{1}},L^{q),\theta
_{2}}\right) }=\left\Vert F_{f}^{p),\theta _{1}}\right\Vert _{q),\theta
_{2}}=\left\Vert \left\Vert f.\chi _{Q+x}\right\Vert _{p),\theta
_{1}}\right\Vert _{q),\theta _{2}}.
\end{equation*}
\end{definition}

Since generalized grand Lebesgue spaces are not translation invariant, this
result reflects to the grand Wiener amalgam spaces. Then the definition of $%
W\left( L^{p),\theta _{1}},L^{q),\theta _{2}}\right) $ is depend of the
choice of $Q.$ We give the following theorem for the independedness.

\begin{theorem}
if $\theta _{2}=0$ then the definition of grand Wiener amalgam space $%
W\left( L^{p),\theta _{1}},L^{q),\theta _{2}}\right) $ is independent of the
choice of $Q,$ i.e., different choices of $Q$ define the same space with
equivalent norms.
\end{theorem}

\begin{proof}
We know that when $\theta _{2}=0,$ the generalized grand Lebesgue space
reduces to lebesgue space $L^{p}\left( \Omega \right) .$ Then if $\theta
_{2}=0,$ the grand Wiener amalgam space $W\left( L^{p),\theta
_{1}},L^{q),\theta _{2}}\right) $ reduces to $W\left( L^{p),\theta
_{1}},L^{q}\right) .$ Since $L^{q}\left( \Omega \right) $ is solid and
strongly translation invariant, by using the proof technic in Proposition
11.3.2 (b) in $\left[ 15\right] $ or Theorem 2 in $\left[ 9\right] ,$ the
proof is completed.
\end{proof}

\begin{theorem}
Let $1<p,q<\infty ,$ and $\theta \geq 0.$ Then $W\left( L^{p)\,,\theta
},L^{q),\theta }\right) $ is a Banach function space (shortly BF-space),
namely its norm satisfies the following properties, where $f,g$ and $f_{n%
\text{ }}$are in $W\left( L^{p),\theta },L^{q),\theta }\right) ,$ $\lambda
\geq 0$ and $E$ is a measurable subset of $\Omega :$
\end{theorem}

$1.$ $\left\Vert f\right\Vert _{W\left( L^{p),\theta },L^{q),\theta }\right)
}\geq 0$

$2.\left\Vert f\right\Vert _{W\left( L^{p),\theta },L^{q),\theta }\right)
}=0 $ if and only if $f=0$ a.e in $\Omega $

$3.\left\Vert \lambda f\right\Vert _{W\left( L^{p),\theta },L^{q),\theta
}\right) }=\lambda \left\Vert f\right\Vert _{W\left( L^{p),\theta
},L^{q),\theta }\right) }$

$4.\left\Vert f+g\right\Vert _{W\left( L^{p),\theta },L^{q),\theta }\right)
}\leq \left\Vert f\right\Vert _{W\left( L^{p),\theta },L^{q),\theta }\right)
}+\left\Vert g\right\Vert _{W\left( L^{p),\theta },L^{q),\theta }\right) }$

$5.$ if $\left\vert g\right\vert \leq \left\vert f\right\vert $ a.e. in $%
\Omega ,$ then $\left\Vert g\right\Vert _{W\left( L^{p),\theta
},L^{q),\theta }\right) }\leq \left\Vert f\right\Vert _{W\left( L^{p),\theta
},L^{q),\theta }\right) }$

$6.$ if $0\leq f_{n}\uparrow f$ a.e. in $\Omega ,$ then $\left\Vert
f_{n}\right\Vert _{W\left( L^{p),\theta },L^{q),\theta }\right) }\uparrow
\left\Vert f\right\Vert _{W\left( L^{p),\theta },L^{q),\theta }\right) }$

$7.\left\Vert \chi _{E}\right\Vert _{W\left( L^{p),\theta },L^{q),\theta
}\right) }<+\infty $

$8.\dint\limits_{E}\left\vert f\right\vert dx\leq C\left( p,\theta ,E\right)
\left\Vert f\right\Vert _{W\left( L^{p),\theta },L^{q),\theta }\right) }$

for some $0<C<\infty .$

\begin{proof}
The first three properties follow from the definition of the norm $%
\left\Vert .\right\Vert _{W\left( L^{p),\theta },L^{q),\theta }\right) }.$

Proof of property $4.$ Let $f,g\in W\left( L^{p),\theta },L^{q),\theta
}\right) .$ Then%
\begin{eqnarray*}
\left\Vert f+g\right\Vert _{W\left( L^{p),\theta },L^{q),\theta }\right) }
&=&\left\Vert \left\Vert \left( f+g\right) \chi _{Q+x}\right\Vert
_{p),\theta }\right\Vert _{q),\theta }=\left\Vert \left\Vert f\chi
_{Q+x}+g\chi _{Q+x}\right\Vert _{p),\theta }\right\Vert _{q),\theta } \\
&\leq &\left\Vert \left\Vert f\chi _{Q+x}\right\Vert _{p),\theta
}+\left\Vert g\chi _{Q+x}\right\Vert _{p),\theta }\right\Vert _{q),\theta
}=\left\Vert F_{f}^{p),\theta }\left( x\right) +F_{g}^{p),\theta }\left(
x\right) \right\Vert _{q),\theta } \\
&\leq &\left\Vert F_{f}^{p),\theta }\left( x\right) \right\Vert _{q),\theta
}+\left\Vert F_{g}^{p),\theta }\left( x\right) \right\Vert _{q),\theta
}=\left\Vert f\right\Vert _{W\left( L^{p),\theta },L^{q),\theta }\right)
}+\left\Vert g\right\Vert _{W\left( L^{p),\theta },L^{q),\theta }\right) }.
\end{eqnarray*}

Proof of property $5.$ Let $\left\vert g\right\vert \leq \left\vert
f\right\vert .$ Then $\left\vert g\right\vert \chi _{Q+x}\leq \left\vert
f\right\vert \chi _{Q+x}$ a.e. in $\Omega .$ Since $L^{p),\theta }\left(
\Omega \right) $ is a BF-space on $\Omega ,$ we have 
\begin{equation*}
F_{g}^{p),\theta }\left( x\right) =\left\Vert g\chi _{Q+x}\right\Vert
_{p),\theta }\leq \left\Vert f\chi _{Q+x}\right\Vert _{p),\theta
}=F_{f}^{p),\theta }\left( x\right) .
\end{equation*}%
Thus we obtain $\left\Vert g\right\Vert _{W\left( L^{p),\theta
},L^{q),\theta }\right) }\leq \left\Vert f\right\Vert _{W\left( L^{p),\theta
},L^{q),\theta }\right) }.$

Proof of property $6.$ Since $0\leq f_{n}\uparrow f$ a.e. in $\Omega ,$ then 
$f_{n}\chi _{Q+x}\uparrow f\chi _{Q+x}$ a.e. in $\Omega .$ Then 
\begin{equation*}
F_{f_{n}}^{p),\theta }\left( x\right) =\left\Vert f_{n}\chi
_{Q+x}\right\Vert _{p),\theta }\uparrow \left\Vert f\chi _{Q+x}\right\Vert
_{p),\theta }=F_{f}^{p),\theta }\left( x\right)
\end{equation*}%
by proposition 2.1, in $\left[ 1\right] .$ One more applaying this
proposition we have 
\begin{equation*}
\left\Vert f_{n}\right\Vert _{W\left( L^{p),\theta },L^{q),\theta }\right)
}=\left\Vert \left\Vert f_{n}\chi _{Q+x}\right\Vert _{p),\theta }\right\Vert
_{q),\theta }\uparrow \left\Vert \left\Vert f\chi _{Q+x}\right\Vert
_{p),\theta }\right\Vert _{q),\theta }=\left\Vert f\right\Vert _{W\left(
L^{p),\theta },L^{q),\theta }\right) }.
\end{equation*}

Proof of property $7.$ Since $\varepsilon ^{\frac{\theta }{p-\varepsilon }}$
is increasing in $\left[ 0,p-1\right] ,$ and $\mu \left( \Omega \right) $ is
finite, then 
\begin{eqnarray}
\left\Vert \chi _{E}\chi _{Q+x}\right\Vert _{W\left( L^{p),\theta
},L^{q),\theta }\right) } &=&\left\Vert \chi _{E\cap Q+x}\right\Vert
_{W\left( L^{p),\theta },L^{q),\theta }\right) }=\left\Vert \left\Vert \chi
_{E\cap Q+x}\right\Vert _{p),\theta }\right\Vert _{q),\theta }  \TCItag{1} \\
&=&\left\Vert \sup_{0<\varepsilon \leq p-1}\varepsilon ^{\frac{\theta }{%
p-\varepsilon }}\left\Vert \chi _{E\cap Q+x}\right\Vert _{p-\varepsilon
}\right\Vert _{q),\theta }  \notag \\
&=&\left\Vert \sup_{0<\varepsilon \leq p-1}\varepsilon ^{\frac{\theta }{%
p-\varepsilon }}\left\{ \int_{\Omega }\left\vert \chi _{E\cap
Q+x}\right\vert ^{\frac{1}{p-\varepsilon }}d\mu \right\} ^{\frac{1}{%
p-\varepsilon }}\right\Vert _{q),\theta }  \notag \\
&\leq &\left\Vert \sup_{0<\varepsilon \leq p-1}\varepsilon ^{\frac{\theta }{%
p-\varepsilon }}\left[ \mu \left( \Omega \right) \right] ^{\frac{1}{%
p-\varepsilon }}\right\Vert _{q),\theta }.  \notag
\end{eqnarray}

If $\mu \left( \Omega \right) <1,$ then $\mu \left( \Omega \right) ^{\frac{1%
}{p-\varepsilon }}<\mu \left( \Omega \right) ^{\frac{1}{p}}$ for all $%
0<\varepsilon \leq p-1.$ Since $\varepsilon ^{\frac{\theta }{p-\varepsilon }%
} $ is increasing, by $\left( 1\right) $ we have 
\begin{eqnarray*}
\left\Vert \chi _{E}\chi _{Q+x}\right\Vert _{W\left( L^{p),\theta
},L^{q),\theta }\right) } &\leq &\left\Vert \sup_{0<\varepsilon \leq
p-1}\varepsilon ^{\frac{\theta }{p-\varepsilon }}\left[ \mu \left( \Omega
\right) \right] ^{\frac{1}{p-\varepsilon }}\right\Vert _{q),\theta } \\
&\leq &\left\Vert \left( p-1\right) ^{\theta }\left[ \mu \left( \Omega
\right) \right] ^{\frac{1}{p}}\right\Vert _{q),\theta } \\
&=&\sup_{0<\eta \leq q-1}\eta ^{\frac{\theta }{q-\eta }}\left\Vert \left(
p-1\right) ^{\theta }\left[ \mu \left( \Omega \right) \right] ^{\frac{1}{p}%
}\right\Vert _{q-\eta } \\
&=&\left( p-1\right) ^{\theta }\left( q-1\right) ^{\theta }\left[ \mu \left(
\Omega \right) \right] ^{\frac{1}{p}}\left[ \mu \left( \Omega \right) \right]
^{\frac{1}{q}} \\
&=&\left( p-1\right) ^{\theta }\left( q-1\right) ^{\theta }\left[ \mu \left(
\Omega \right) \right] ^{\frac{1}{p}+\frac{1}{q}}<\infty .
\end{eqnarray*}%
If $\mu \left( \Omega \right) >1,$ then $\mu \left( \Omega \right) ^{\frac{1%
}{p-\varepsilon }}>\mu \left( \Omega \right) ^{\frac{1}{p}}$ for all $%
0<\varepsilon \leq p-1.$ Again from $\left( 1\right) ,$ we have 
\begin{eqnarray*}
\left\Vert \chi _{E}\chi _{Q+x}\right\Vert _{W\left( L^{p),\theta
},L^{q),\theta }\right) } &\leq &\left\Vert \sup_{0<\varepsilon \leq
p-1}\varepsilon ^{\frac{\theta }{p-\varepsilon }}\left[ \mu \left( \Omega
\right) \right] ^{\frac{1}{p-\varepsilon }}\right\Vert _{q),\theta } \\
&=&\left\Vert \left( p-1\right) ^{\theta }\left[ \mu \left( \Omega \right) %
\right] \right\Vert _{q),\theta }=\sup_{0<\eta \leq q-1}\eta ^{\frac{\theta 
}{q-\eta }}\left\Vert \left( p-1\right) ^{\theta }\mu \left( \Omega \right)
\right\Vert _{q-\eta } \\
&\leq &\left( p-1\right) ^{\theta }\left( q-1\right) ^{\theta }\left[ \mu
\left( \Omega \right) \right] ^{2}<\infty .
\end{eqnarray*}

Proof of property $8.$ \ Fix $0<\varepsilon \leq p-1.$ By the H\"{o}lder's
inequality for Wiener's amalgam space we have 
\begin{equation*}
\dint\limits_{E}\left\vert f\right\vert dx=\dint\limits_{\Omega }\left\vert
f\chi _{E}\right\vert dx\leq \left\Vert f\right\Vert _{W\left(
L^{p-\varepsilon },L^{q-\varepsilon }\right) }\left\Vert \chi
_{E}\right\Vert _{W\left( L^{\left( p-\varepsilon \right) ^{\prime
}},L^{\left( q-\varepsilon \right) ^{\prime }}\right) }=C\left\Vert
f\right\Vert _{W\left( L^{p-\varepsilon },L^{q-\varepsilon }\right) }
\end{equation*}%
where $\frac{1}{p-\varepsilon }+\frac{1}{\left( p-\varepsilon \right)
^{^{\prime }}}=\frac{1}{q}+\frac{1}{\left( q-\varepsilon \right) ^{^{\prime
}}}=1,$ and $C=C\left( p,\theta ,E\right) =$\ $\left\Vert \chi
_{E}\right\Vert _{W\left( L^{\left( p-\varepsilon \right) ^{\prime
}},L^{\left( q-\varepsilon \right) ^{\prime }}\right) }.$ Thus we obtain 
\begin{equation*}
\dint\limits_{E}\left\vert f\right\vert dx\leq C\sup_{0<\varepsilon \leq
p-1}\varepsilon ^{\frac{\theta }{p-\varepsilon }}\left\Vert f\right\Vert
_{W\left( L^{p-\varepsilon },L^{q-\varepsilon }\right) }=C\left\Vert
f\right\Vert _{W\left( L^{p)},L^{q)}\right) }.
\end{equation*}
\end{proof}

\begin{theorem}
The grand Wiener amalgam space $W\left( L^{p),\theta },L^{q),\theta }\right) 
$ is a Banach space.
\end{theorem}

\begin{proof}
To show that $W\left( L^{p),\theta },L^{q),\theta }\right) $ is a Banach
space it is suffices to show that if $\left( f_{n}\right) _{n\in \mathbb{N}}$
is a sequence in $W\left( L^{p),\theta },L^{q),\theta }\right) $ with 
\begin{equation*}
\dsum\limits_{n\in \mathbb{N}}\left\Vert f_{n}\right\Vert _{W\left(
L^{p),\theta },L^{q),\theta }\right) }<\infty ,
\end{equation*}%
then $\dsum\limits_{n\in \mathbb{N}}f_{n}$ converges to an element of $%
W\left( L^{p),\theta },L^{q),\theta }\right) .$ The proof of this is mutadis
and mutandis \ same as in the proof of completness of Wiener amalgam space $%
W\left( L^{p},L^{q}\right) ,$ (see Proposition 11.3.2, in $\left[ 15\right]
, $ and $\left[ 9\right] ).$
\end{proof}

\section{\textbf{Inclusions and consequences}}

\begin{proposition}
Let $1<p,q<\infty $ . Then for an arbitrary $\varepsilon $ and $\eta ,$ $%
0<\varepsilon \leq p-1,$ $0<\eta \leq q-1,$ we have 
\begin{equation*}
W\left( L^{p},L^{q}\right) \subset W\left( L^{p),\theta },L^{q),\theta
}\right) \subset W\left( L^{p-\varepsilon },L^{q-\eta }\right) .
\end{equation*}%
If $q\leq p,$ then the inclusion $W\left( L^{p},L^{q}\right) \subset W\left(
L^{p),\theta },L^{q),\theta }\right) $ is strict.
\end{proposition}

\begin{proof}
By the definition of generalized grand Lebesgue space we have 
\begin{equation}
F_{f}^{p-\varepsilon }\left( x\right) =\left\Vert f.\chi _{Q+x}\right\Vert
_{p-\varepsilon }\leq \left\Vert f.\chi _{Q+x}\right\Vert _{p),\theta
}=F_{f}^{p),\theta }\left( x\right) .  \tag{2}
\end{equation}%
Then from $\left( 2\right) ,$ 
\begin{align*}
\left\Vert f\right\Vert _{W\left( L^{p-\varepsilon },L^{q-\eta }\right) }&
=\left\Vert \left\Vert f.\chi _{Q+x}\right\Vert _{p-\varepsilon }\right\Vert
_{q-\eta }\leq \left\Vert F_{f}^{p),\theta }\left( x\right) \right\Vert
_{q-\eta } \\
& \leq \sup_{0<\eta \leq q-1}\varepsilon ^{\frac{\theta }{q-\eta }%
}\left\Vert F_{f}^{p),\theta }\left( x\right) \right\Vert _{q),\theta
}=\left\Vert f\right\Vert _{W\left( L^{p),\theta },L^{q,\theta )}\right) }.
\end{align*}%
This implies $W\left( L^{p),\theta },L^{q),\theta }\right) \subset W\left(
L^{p-\varepsilon },L^{q-\eta }\right) .$

\qquad We now want to show that $W\left( L^{p},L^{q}\right) \subset W\left(
L^{p),\theta },L^{q),\theta }\right) .$ It is known that $L^{p}\subset
L^{p),\theta }$ and $L^{q}\subset L^{q),\theta }.$ Then there exists $%
D_{1}>0 $ and $D_{2}>0$ such that 
\begin{equation*}
\left\Vert g\right\Vert _{p),\theta }\leq D_{1}\left\Vert g\right\Vert _{p},
\end{equation*}%
\begin{equation}
\left\Vert h\right\Vert _{q),\theta }\leq D_{2}\left\Vert h\right\Vert _{q}.
\tag{3}
\end{equation}%
Then by $\left( 2\right) $ and $\left( 3\right) $%
\begin{equation}
F_{f}^{p),\theta }\left( x\right) =\left\Vert f.\chi _{Q+x}\right\Vert
_{p),\theta }\leq D_{1}\left\Vert f.\chi _{Q+x}\right\Vert
_{p}=D_{1}F_{f}^{p}\left( x\right) .  \tag{4}
\end{equation}%
Hence by $\left( 3\right) $ and $\left( 4\right) ,$%
\begin{equation*}
\left\Vert f\right\Vert _{W\left( L^{p),\theta },L^{q),\theta }\right)
}=\left\Vert F_{f}^{p),\theta }\right\Vert _{q),\theta }\leq \left\Vert
D_{1}F_{f}^{p}\right\Vert _{q),\theta }\leq D_{1}D_{2}\left\Vert
F_{f}^{p}\right\Vert _{q}=\left\Vert f\right\Vert _{W\left(
L^{p},L^{q}\right) }.
\end{equation*}%
Thus $W\left( L^{p},L^{q}\right) \subseteq W\left( L^{p),\theta
},L^{q),\theta }\right) .$

\qquad To see that the inclusion $W\left( L^{p},L^{q}\right) \subseteq
W\left( L^{p),\theta },L^{q),\theta }\right) $ is strict, take $\Omega
=\left( 0,1\right) $ and $f\left( x\right) =x^{-\frac{1}{p}},$ for $p>1.$
Then 
\begin{eqnarray*}
F_{f}^{p)}\left( x\right) &=&\left\Vert t^{-\frac{1}{p}}\chi
_{Q+x}\right\Vert _{p),\theta }\leq \left\Vert t^{-\frac{1}{p}}\right\Vert
_{p),\theta }=\sup_{0<\varepsilon \leq p-1}\varepsilon ^{\frac{\theta }{%
p-\varepsilon }}\left( \left\Vert t^{-\frac{1}{p}}\right\Vert
_{p-\varepsilon }\right) \\
&=&\sup_{0<\varepsilon \leq p-1}\varepsilon ^{\frac{\theta }{p-\varepsilon }%
}\left( \int_{0}^{1}t^{-\frac{1}{p}\left( p-\varepsilon \right) }dt\right) ^{%
\frac{1}{p-\varepsilon }}=\sup_{0<\varepsilon \leq p-1}\varepsilon ^{\frac{%
\theta }{p-\varepsilon }}\left( \lim_{a\rightarrow 0}\left( \int_{a}^{1}t^{-%
\frac{1}{p}\left( p-\varepsilon \right) }dt\right) ^{\frac{1}{p-\varepsilon }%
}\right) \\
&=&\sup_{0<\varepsilon \leq p-1}\varepsilon ^{\frac{\theta }{p-\varepsilon }}%
\left[ \lim_{a\rightarrow 0}\left( \frac{p}{\varepsilon }t^{\frac{%
\varepsilon }{p}}\mid _{t=a}^{t=1}\right) \right] ^{\frac{1}{p-\varepsilon }%
}=\sup_{0<\varepsilon \leq p-1}\varepsilon ^{\frac{\theta }{p-\varepsilon }}%
\left[ \lim_{a\rightarrow 0}\left( \frac{p}{\varepsilon }\left( 1-a^{\frac{%
\varepsilon }{p}}\right) \right) \right] ^{\frac{1}{p-\varepsilon }} \\
&=&\sup_{0<\varepsilon \leq p-1}\varepsilon ^{\frac{\theta }{p-\varepsilon }%
}\left( \frac{p}{\varepsilon }\right) ^{\frac{1}{p-\varepsilon }%
}=\sup_{0<\varepsilon \leq p-1}\varepsilon ^{\frac{\theta -1}{p-\varepsilon }%
}.p^{\frac{1}{p-\varepsilon }}.
\end{eqnarray*}%
Since $\theta \geq 1,$ 
\begin{equation}
F_{f}^{p)}\left( x\right) =\sup_{0<\varepsilon \leq p-1}\varepsilon ^{\frac{%
\theta -1}{p-\varepsilon }}.p^{\frac{1}{p-\varepsilon }}\leq \left(
p-1\right) ^{\theta -1}p<p^{\theta }.  \tag{5}
\end{equation}%
Thus from $\left( 5\right) ,$ 
\begin{eqnarray*}
\left\Vert f\right\Vert _{W\left( L^{p)},L^{q)}\right) } &=&\left\Vert
F_{f}^{p)}\right\Vert _{q),\theta }<\left\Vert p^{\theta }\right\Vert
_{q),\theta } \\
&=&\sup_{0<\eta \leq q-1}\eta ^{\frac{\theta }{q-\eta }}\left\{
\int_{0}^{1}\left( p^{\theta }\right) ^{q-\eta }dx\right\} ^{\frac{1}{q-\eta 
}}=\sup_{0<\eta \leq q-1}\varepsilon ^{\frac{\theta }{q-\eta }}p^{\theta } \\
&=&\left( q-1\right) ^{\theta }p^{\theta }<\infty .
\end{eqnarray*}%
Then $f\left( x\right) =x^{-\frac{1}{p}}\in W\left( L^{p),\theta
},L^{q),\theta }\right) \left( 0,1\right) .$ In the other hand it is easy to
show that $f\left( t\right) =t^{-\frac{1}{p}}\notin L^{p}\left( 0,1\right) .$
Since $W\left( L^{p},L^{q}\right) \left( 0,1\right) \subset L^{p}\left(
0,1\right) $ for $q\leq p,$ then $f\left( t\right) =t^{-\frac{1}{t}}\notin
W\left( L^{p},L^{q}\right) \left( 0,1\right) .$ So we have $f\left( t\right)
=t^{-\frac{1}{p}}\in W\left( L^{p),\theta },L^{q),\theta }\right) \backslash
W\left( L^{p},L^{q}\right) .$ Thus if $q\leq p,$ the inclusion $W\left(
L^{p},L^{q}\right) \subset W\left( L^{p),\theta },L^{q),\theta }\right) $ is
strict.
\end{proof}

\begin{proposition}
If $1<p_{2}\leq p_{1}<\infty $ and $1<q<\infty $, then $W\left(
L^{p_{1}),\theta },L^{q),\theta }\right) \subset W\left( L^{p_{2}),\theta
},L^{q),\theta }\right) .$
\end{proposition}

\begin{proof}
Let $f\in W\left( L^{p_{1}),\theta },L^{q),\theta }\right) .$ Then $f\chi
_{Q+x}\in L^{p_{1}),\theta }$. Since $p_{2}\leq p_{1},$ by Theorem 3, in $%
\left[ 15\right] $ we have $L^{p_{1}),\theta }\subset L^{p_{2}),\theta },$
and 
\begin{equation*}
\left\Vert f\chi _{Q+x}\right\Vert _{^{p_{2}),\theta }}\leq C\left\Vert
f\chi _{Q+x}\right\Vert _{^{p_{1}),\theta }},\text{ }x\in \Omega
\end{equation*}%
for some $C>0.$ Thus by the solidness of $L^{q),\theta },$ 
\begin{eqnarray*}
\left\Vert f\right\Vert _{W\left( L^{p_{2}),\theta },L^{q),\theta }\right) }
&=&\left\Vert \left\Vert f.\chi _{Q+x}\right\Vert _{^{p_{2}),\theta
}}\right\Vert _{\left\Vert f\right\Vert _{^{q),\theta }}} \\
&\leq &C\left\Vert \left\Vert f.\chi _{Q+x}\right\Vert _{^{p_{1}),\theta
}}\right\Vert _{q),\theta }=\left\Vert f\right\Vert _{W\left(
L^{p_{1}),\theta },L^{q),\theta }\right) },
\end{eqnarray*}%
and we have\ \ $W\left( L^{p_{1}),\theta },L^{q),\theta }\right) \subset
W\left( L^{p_{2}),\theta },L^{q),\theta }\right) .$
\end{proof}

\begin{proposition}
If $1<q_{2}\leq q_{1}<\infty $ and $1<p<\infty $, then $W\left( L^{p),\theta
},L^{q_{1}),\theta }\right) \subset W\left( L^{p),\theta },L^{q_{2}),\theta
)}\right) .$
\end{proposition}

\begin{proof}
Let $f\in W\left( L^{p),\theta },L^{q_{1}),\theta }\right) .$ Then $%
F_{f}^{p)}\in L^{q_{1}),\theta }.$ Since $1<q_{2}\leq q_{1}<\infty ,$by
Theorem 3 in $\left[ 14\right] ,$ there exists $C>0$ such that 
\begin{equation}
\left\Vert F_{f}^{p),\theta }\right\Vert _{W\left( L^{p),\theta
},L^{q_{2}),\theta }\right) }\leq C\left\Vert F_{f}^{p),\theta }\right\Vert
_{W\left( L^{p),\theta },L^{q_{1}),\theta }\right) }<\infty  \tag{6}
\end{equation}%
for all $\theta >0.$ Then $F_{f}^{p),\theta }\in L^{q_{2}),\theta }.$ This
implies $f\in W\left( L^{p),\theta },L^{q_{2}),\theta }\right) .$ Thus we
have the inclusion 
\begin{equation*}
W\left( L^{p),\theta },L^{q_{1}),\theta }\right) \subset W\left(
L^{p),\theta },L^{q_{2}),\theta )}\right) .
\end{equation*}
\end{proof}

By using the Proposition 2 and proposition 3, we easily prove the following
corollary.

\begin{corollary}
If $1<p_{2}\leq p_{1}<\infty $ and $1<q_{2}\leq q_{1}<\infty $, then $%
W\left( L^{p_{1}),\theta },L^{q_{1}),\theta }\right) \subset W\left(
L^{p_{2}),\theta },L^{q_{2}),\theta }\right) .$
\end{corollary}

\begin{proposition}
Let $1<p_{i},q_{i}<\infty ,\left( i=1,2,3\right) .$ If there exist constants 
$C_{1}>0,C_{1}>0,$ such that for all $u\in L^{p_{1}),\theta },v\in
L^{p_{2}),}$\ 
\begin{equation}
\left\Vert uv\right\Vert _{L^{p_{3}),\theta }}\leq C_{1}\left\Vert
u\right\Vert _{L^{p_{1}),\theta }}\left\Vert v\right\Vert _{L^{p_{2}),\theta
}}  \tag{7}
\end{equation}%
and for all $u\in L^{q_{1}),\theta },v\in L^{q_{2}),\theta }$\ \ 
\begin{equation}
\left\Vert uv\right\Vert _{L^{q_{3}),\theta }}\leq C_{2}\left\Vert
u\right\Vert _{L^{q_{1}),\theta }}\left\Vert v\right\Vert _{L^{q_{2}),\theta
}},  \tag{8}
\end{equation}%
then there exists a constant $C>0,$ such that for all $f\in W\left(
L^{p_{1}),\theta },L^{q_{1}),\theta }\right) $ and $g\in W\left(
L^{p_{2}),\theta },L^{q_{2}),\theta }\right) ,$ we have $fg\in W\left(
L^{p_{3}),\theta },L^{q_{3}),\theta }\right) $ and 
\begin{equation*}
\left\Vert f.g\right\Vert _{W\left( L^{p_{3}),\theta },L^{q_{3}),\theta
}\right) }\leq C_{1}C_{2}\left\Vert f\right\Vert _{W\left( L^{p_{1}),\theta
},L^{q_{1}),\theta }\right) }\left\Vert g\right\Vert _{W\left(
L^{p_{2}),\theta },L^{q_{2}),\theta }\right) }.
\end{equation*}
\end{proposition}

\begin{proof}
Let \ $f\in W\left( L^{p_{1}),\theta },L^{q_{1}),\theta }\right) $ and $g\in
W\left( L^{p_{2}),\theta },L^{q_{2}),\theta }\right) .$ Since $\chi
_{Q+x}^{2}=\chi _{Q+x},$ from $\left( 7\right) $ and $\left( 8\right) ,$%
\begin{eqnarray*}
\left\Vert fg\right\Vert _{W\left( L^{p_{3}),\theta },L^{q_{3}),\theta
}\right) } &=&\left\Vert \left\Vert \left( fg\right) \chi _{Q+x}\right\Vert
_{L^{p_{3}),\theta }}\right\Vert _{L^{q_{3}),\theta }} \\
&=&\left\Vert \left\Vert \left( f\chi _{Q+x}\right) \left( f\chi
_{Q+x}\right) \right\Vert _{L^{p_{3}),\theta }}\right\Vert
_{L^{q_{3})},\theta } \\
&\leq &C_{1}\left\Vert \left\Vert f\chi _{Q+x}\right\Vert _{L^{p_{1}),\theta
}}\left\Vert f\chi _{Q+x}\right\Vert _{L^{p_{2}),\theta }}\right\Vert
_{L^{q_{3}),\theta }} \\
&\leq &C_{1}C_{2}\left\Vert \left\Vert f\chi _{Q+x}\right\Vert
_{L^{p_{1}),\theta }}\right\Vert _{L^{q_{1}),\theta }}\left\Vert \left\Vert
f\chi _{Q+x}\right\Vert _{L^{p_{2}),\theta }}\right\Vert _{L^{q_{2}),\theta
}} \\
&=&C_{1}C_{2}\left\Vert f\right\Vert _{W\left( L^{p_{1}),\theta
},L^{q_{1}),\theta }\right) }\left\Vert g\right\Vert _{W\left(
L^{p_{2}),\theta },L^{q_{2}),\theta }\right) }.
\end{eqnarray*}
\end{proof}

\begin{proposition}
Let $1<p\leq \infty .$ Then $W\left( L^{p),\theta },L^{p),\theta }\right)
\left( \Omega \right) =L^{p),\theta }\left( \Omega \right) .$
\end{proposition}

\begin{proof}
According to the definition of the supremum, for an arbitrary $\eta >0,$
there exists $0<$ $\varepsilon _{0}\leq p-1,$ such that for all $f\in
L^{p),\theta }\left( \Omega \right) ,$ 
\begin{equation}
\sup_{0<\varepsilon \leq p-1}\left( \varepsilon ^{\frac{\theta }{%
p-\varepsilon }}\left\Vert f\right\Vert _{p-\varepsilon }\right) \leq
\varepsilon ^{\frac{\theta }{p-\varepsilon _{0}}}\left\Vert f\right\Vert
_{p-\varepsilon _{0}}+\eta .  \tag{9}
\end{equation}%
Then we have 
\begin{eqnarray}
F_{f}^{p),\theta }\left( x\right) &=&\left\Vert f.\chi _{Q+x}\right\Vert
_{p),\theta }=\sup_{0<\varepsilon \leq p-1}\left( \varepsilon ^{\frac{\theta 
}{p-\varepsilon }}\left\Vert f.\chi _{Q+x}\right\Vert _{p-\varepsilon
}\right)  \TCItag{10} \\
&\leq &\left( \varepsilon _{0}^{\frac{\theta }{p-\varepsilon _{0}}%
}\left\Vert f.\chi _{Q+x}\right\Vert _{p-\varepsilon _{0}}+\eta \right)
=\varepsilon ^{\frac{\theta }{p-\varepsilon _{0}}}F_{f}^{p-\varepsilon
_{0}}+\eta ,  \notag
\end{eqnarray}%
and so%
\begin{equation}
F_{f}^{p),\theta }\left( x\right) =\sup_{0<\varepsilon \leq p-1}\left(
\varepsilon ^{\frac{\theta }{p-\varepsilon }}\left\Vert F_{f}\right\Vert
_{p-\varepsilon }\right) \leq \varepsilon _{0}^{\frac{\theta }{p-\varepsilon
_{0}}}\left\Vert F_{f}\right\Vert _{p-\varepsilon _{0}}+\eta .  \tag{11}
\end{equation}%
Thus from $\left( 10\right) $ and $\left( 11\right) $ we have%
\begin{eqnarray*}
\left\Vert f\right\Vert _{W\left( L^{p),\theta },L^{p),\theta }\right) }
&=&\left\Vert F_{f}^{p),\theta }\right\Vert _{p),\theta
}=\sup_{0<\varepsilon \leq p-1}\varepsilon ^{\frac{\theta }{p-\varepsilon }%
}\left\Vert F_{f}^{p),\theta }\right\Vert _{p-\varepsilon }\leq \varepsilon
_{0}^{\frac{\theta }{p-\varepsilon _{0}}}\left\Vert F_{f}^{p),\theta
}\right\Vert _{p-\varepsilon _{0}}+\eta \\
&\leq &\varepsilon _{0}^{\frac{\theta }{p-\varepsilon _{0}}}\left\Vert
\varepsilon _{0}^{\frac{\theta }{p-\varepsilon _{0}}}\left\Vert f.\chi
_{Q+x}\right\Vert _{p-\varepsilon _{0}}+\eta \right\Vert _{p-\varepsilon
_{0}}+\eta \\
&\leq &\varepsilon _{0}^{\frac{\theta }{p-\varepsilon _{0}}}\left\Vert
\varepsilon _{0}^{\frac{\theta }{p-\varepsilon _{0}}}\left\Vert f.\chi
_{Q+x}\right\Vert _{p-\varepsilon _{0}}\right\Vert _{p-\varepsilon
_{0}}+\varepsilon ^{\frac{\theta }{p-\varepsilon _{0}}}\left\Vert \eta
\right\Vert _{p-\varepsilon _{0}}+\eta \\
&=&\varepsilon _{0}^{\frac{\theta }{p-\varepsilon _{0}}}\varepsilon _{0}^{%
\frac{\theta }{p-\varepsilon _{0}}}\left\Vert f\right\Vert _{W\left(
L^{p-\varepsilon _{0}},L^{p-\varepsilon _{0}}\right) }+\left\Vert \eta
\right\Vert _{p-\varepsilon _{0}}+\eta \\
&=&\varepsilon _{0}^{\frac{\theta }{p-\varepsilon _{0}}}\left( \varepsilon
_{0}^{\frac{\theta }{p-\varepsilon _{0}}}\left\Vert f\right\Vert
_{p-\varepsilon _{0}}\right) +\eta .\mu \left( \Omega \right) +\eta .
\end{eqnarray*}%
If $\eta \rightarrow 0,$ the right side of $\left( 11\right) $ approaches to 
$C_{1}\left\Vert f\right\Vert _{W\left( L^{p),\theta },L^{p),\theta }\right)
}$ and we have%
\begin{equation}
\left\Vert f\right\Vert _{W\left( L^{p),\theta },L^{p),\theta }\right) }\leq
C_{1}\left\Vert f\right\Vert _{p),\theta }  \tag{12}
\end{equation}%
for some constant $C_{1}>0.$

\qquad Conversely let $f\in L^{p),\theta }.$ For this $\eta >0,$we obtain%
\begin{eqnarray*}
\left\Vert f\right\Vert _{p),\theta } &=&\sup_{0<\varepsilon \leq
p-1}\varepsilon ^{\frac{\theta }{p-\varepsilon }}\left\Vert f\right\Vert
_{p-\varepsilon }\leq \varepsilon _{0}^{\frac{\theta }{p-\varepsilon _{0}}%
}\left\Vert f\right\Vert _{p-\varepsilon _{0}}+\eta \\
&=&\varepsilon _{0}^{\frac{\theta }{p-\varepsilon _{0}}}\left\Vert
f\right\Vert _{W\left( L^{p-\varepsilon _{0}},L^{p-\varepsilon _{0}}\right)
}+\eta \leq \\
&=&\varepsilon _{0}^{\frac{\theta }{p-\varepsilon _{0}}}\left\Vert
\left\Vert f.\chi _{Q+x}\right\Vert _{p-\varepsilon _{0}}\right\Vert
_{p-\varepsilon _{0}}+\eta \\
&\leq &\left\Vert \varepsilon _{0}^{\frac{\theta }{p-\varepsilon _{0}}%
}\left\Vert f.\chi _{Q+x}\right\Vert _{p-\varepsilon _{0}}+\eta \right\Vert
_{p-\varepsilon _{0}}+\eta \\
&\leq &\left\Vert \varepsilon _{0}^{\frac{\theta }{p-\varepsilon _{0}}%
}\left\Vert f.\chi _{Q+x}\right\Vert _{p-\varepsilon _{0}}\right\Vert
_{p-\varepsilon _{0}}+\varepsilon _{0}^{\frac{\theta }{p-\varepsilon _{0}}%
}\left\Vert \eta \right\Vert _{p-\varepsilon _{0}}+\eta \\
&=&\varepsilon _{0}^{\frac{-\theta }{p-\varepsilon _{0}}}\left( \varepsilon
_{0}^{\frac{\theta }{p-\varepsilon _{0}}}\left\Vert \varepsilon _{0}^{\frac{%
\theta }{p-\varepsilon _{0}}}\left\Vert f.\chi _{Q+x}\right\Vert
_{p-\varepsilon _{0}}\right\Vert _{p-\varepsilon _{0}}\right) +\varepsilon
_{0}^{\frac{\theta }{p-\varepsilon _{0}}}\left\Vert \eta \right\Vert
_{p-\varepsilon _{0}}+\eta
\end{eqnarray*}%
If $\eta \rightarrow 0,$ then the right side of $\left( 12\right) $
approaches to $C_{2}\left\Vert f\right\Vert _{W\left( L^{p),\theta
},L^{p),\theta }\right) }$ and we have%
\begin{equation}
\left\Vert f\right\Vert _{p),\theta }\leq C_{2}\left\Vert f\right\Vert
_{W\left( L^{p),\theta },L^{p),\theta }\right) }  \tag{13}
\end{equation}%
for some constant $C_{2}>0.$ Combining $\left( 12\right) $ and $\left(
13\right) ,$ we obtain $W\left( L^{p),\theta },L^{p),\theta }\right) \left(
\Omega \right) =L^{p),\theta }\left( \Omega \right) .$
\end{proof}

\section{H\"{o}lder's inequality, Duality and reflexivity in grand Wiener
amalgam spaces}

It is known by Theorem 2, that $\left\Vert .\right\Vert _{W\left(
L^{p),\theta },L^{q),\theta }\right) }$ is a Banach function norm and $%
W\left( L^{p),\theta },L^{q),\theta }\right) \left( \Omega \right) $ is a
Banach function space.

\begin{definition}
The associate space of \ $W\left( L^{p),\theta },L^{q),\theta }\right)
\left( \Omega \right) $ determined by the associate norm $\left\Vert
.\right\Vert _{W\left( L^{p),\theta },L^{q),\theta }\right) ^{^{\prime }}}$
is 
\begin{equation*}
W\left( L^{p),\theta },L^{q),\theta }\right) ^{^{\prime }}\left( \Omega
\right) =\left\{ g\in \mathcal{M}_{0}:\left\Vert g\right\Vert _{W\left(
L^{p),\theta },L^{q),\theta }\right) ^{^{\prime }}}<\infty \right\} ,
\end{equation*}%
where 
\begin{equation*}
\left\Vert g\right\Vert _{W\left( L^{p),\theta },L^{q),\theta }\right)
^{^{\prime }}}=\sup \left\{ \dint\limits_{\Omega }\left\vert fg\right\vert
d\mu :f\in \mathcal{M}_{0}\left( \Omega \right) ,\left\Vert f\right\Vert
_{W\left( L^{p),\theta },L^{q),\theta }\right) }\leq 1\right\} .
\end{equation*}
\end{definition}

\begin{theorem}
(\textbf{H\"{o}lder's inequality}) If $f\in W\left( L^{p),\theta
},L^{q),\theta }\right) \left( \Omega \right) $ and $g\in W\left(
L^{p),\theta },L^{q),\theta }\right) ^{^{\prime }}\left( \Omega \right) ,$
then $fg$ is integrable and 
\begin{equation}
\dint\limits_{\Omega }f\left( x\right) .g\left( x\right) dx\leq \left\Vert
f\right\Vert _{W\left( L^{p)},L^{q)}\right) }.\left\Vert g\right\Vert
_{W\left( L^{p),\theta },L^{q),\theta }\right) ^{^{\prime }}}.  \tag{14}
\end{equation}
\end{theorem}

\begin{proof}
Since $W\left( L^{p),\theta },L^{q),\theta }\right) \left( \Omega \right) $
is a Banach function space by Theorem 2, the proof is completed by Theorem
2.4., in $\left[ 3\right] .$
\end{proof}

\begin{proposition}
The closure $\overline{C_{0}^{\infty }}\mid _{W\left( L^{p),\theta
},L^{q),\theta }\right) }$of the set $C_{0}^{\infty }$ in the space $W\left(
L^{p),\theta },L^{q),\theta }\right) \left( \Omega \right) $ consists of
functions $f\in W\left( L^{p),\theta },L^{q),\theta }\right) \left( \Omega
\right) $ such that 
\begin{equation}
\lim_{\varepsilon \rightarrow 0}\varepsilon ^{\frac{\theta }{p-\varepsilon }%
}\left\Vert f\right\Vert _{W\left( L^{p-\varepsilon },L^{q-\varepsilon
}\right) }=0,  \tag{15}
\end{equation}%
where $C_{0}^{\infty }$ is the space of infinitely differentiable complex
valued functions with compact support.
\end{proposition}

\begin{proof}
First we will show that $\left( 15\right) $ holds for $\overline{%
C_{0}^{\infty }}\mid _{W\left( L^{p),\theta },L^{q),\theta }\right) }.$
Since $C_{0}^{\infty }\left( \Omega \right) $ is dense $L^{p}\left( \Omega
\right) ,$ then $C_{0}^{\infty }\left( \Omega \right) $ is dense in $W\left(
L^{p},L^{q}\right) $ by Theorem 1, in $\left[ 9\right] .$ Let $f\in 
\overline{C_{0}^{\infty }}\mid _{W\left( L^{p),\theta },L^{q),\theta
}\right) }.$ Then there exists a sequence $\left( f_{n}\right) \subset
W\left( L^{p},L^{q}\right) $ such that 
\begin{equation*}
\left\Vert f_{n}-f\right\Vert _{W\left( L^{p},L^{q}\right) }\rightarrow 0.
\end{equation*}%
Thus for given $\eta >0,$ there exists $n_{0}\in 
%TCIMACRO{\U{2115} }%
%BeginExpansion
\mathbb{N}
%EndExpansion
$ such that 
\begin{equation}
\left\Vert f_{n_{0}}-f\right\Vert _{W\left( L^{p},L^{q}\right) }<\frac{\eta 
}{2}.  \tag{16}
\end{equation}%
Since 
\begin{equation*}
\varepsilon ^{\frac{\theta }{p-\varepsilon }}\left\Vert f_{n_{0}}\right\Vert
_{W\left( L^{p-\varepsilon },L^{q-\varepsilon }\right) }=\varepsilon ^{\frac{%
\theta }{p-\varepsilon }}\left\Vert \left\Vert f_{n_{0}}\chi
_{Q+x}\right\Vert _{p-\varepsilon }\right\Vert _{q-\varepsilon },
\end{equation*}%
and $f_{n_{0}}\chi _{Q+x}\in L^{p},$ by the H\"{o}lder's inequality%
\begin{equation}
\varepsilon ^{\frac{\theta }{p-\varepsilon }}\left\Vert f_{n_{0}}\chi
_{Q+x}\right\Vert _{p-\varepsilon }=\varepsilon ^{\frac{\theta }{%
p-\varepsilon }}\left\{ \dint\limits_{\Omega }\left\vert f_{n_{0}}\left(
t\right) \chi _{Q+x}\right\vert ^{p-\varepsilon }dt\right\} ^{\frac{1}{%
p-\varepsilon }}  \tag{17}
\end{equation}%
\begin{eqnarray}
&\leq &\varepsilon ^{\frac{\theta }{p-\varepsilon }}\left\Vert f_{n_{0}}\chi
_{Q+x}\right\Vert _{p}\mu \left( \Omega \right) ^{\frac{1}{p-\varepsilon }-%
\frac{1}{p}}  \notag \\
&=&\varepsilon ^{\frac{\theta }{p-\varepsilon }}\mu \left( \Omega \right) ^{%
\frac{\varepsilon }{p\left( p-\varepsilon \right) }}\left\Vert f_{n_{0}}\chi
_{Q+x}\right\Vert _{p}.  \notag
\end{eqnarray}%
Thus from $\left( 17\right) ,$ 
\begin{eqnarray*}
\left\Vert \varepsilon ^{\frac{\theta }{p-\varepsilon }}\left\Vert
f_{n_{0}}\chi _{Q+x}\right\Vert _{p-\varepsilon }\right\Vert _{q-\varepsilon
} &\leq &\left\Vert \varepsilon ^{\frac{\theta }{p-\varepsilon }}\mu \left(
\Omega \right) ^{\frac{\varepsilon }{p\left( p-\varepsilon \right) }%
}\left\Vert f_{n_{0}}\chi _{Q+x}\right\Vert _{p}\right\Vert _{q-\varepsilon }
\\
&=&\varepsilon ^{\frac{\theta }{p-\varepsilon }}\mu \left( \Omega \right) ^{%
\frac{\varepsilon }{p\left( p-\varepsilon \right) }}\left\Vert \left\Vert
f_{n_{0}}\chi _{Q+x}\right\Vert _{p}\right\Vert _{q-\varepsilon } \\
&\leq &\varepsilon ^{\frac{\theta }{p-\varepsilon }}\mu \left( \Omega
\right) ^{\frac{\varepsilon }{p\left( p-\varepsilon \right) }}\mu \left(
\Omega \right) ^{\frac{1}{q-\varepsilon }-\frac{1}{q}}\left\Vert \left\Vert
f_{n_{0}}\chi _{Q+x}\right\Vert _{p}\right\Vert _{q} \\
&=&\varepsilon ^{\frac{\theta }{p-\varepsilon }}\mu \left( \Omega \right) ^{%
\frac{\varepsilon }{p\left( p-\varepsilon \right) }+\frac{\varepsilon }{%
q\left( q-\varepsilon \right) }}\left\Vert f_{n_{0}}\right\Vert _{W\left(
L^{p},L^{q}\right) }\rightarrow 0,\text{ }
\end{eqnarray*}%
as $\varepsilon \rightarrow 0.$ Hence there exists $\varepsilon _{0}$ such
that when $\varepsilon <\varepsilon _{0},$%
\begin{equation}
\varepsilon ^{\frac{\theta }{p-\varepsilon }}\left\Vert f_{n_{0}}\right\Vert
_{_{W\left( L^{p-\varepsilon },L^{q-\varepsilon }\right) }}=\left\Vert
\varepsilon ^{\frac{\theta }{p-\varepsilon }}\left\Vert f_{n_{0}}\chi
_{Q+x}\right\Vert _{p-\varepsilon }\right\Vert _{q-\varepsilon }\leq \frac{%
\eta }{2}.  \tag{18}
\end{equation}%
Then by $\left( 16\right) $ and $\left( 18\right) ,$%
\begin{eqnarray*}
\varepsilon ^{\frac{\theta }{p-\varepsilon }}\left\Vert f\right\Vert
_{W\left( L^{p-\varepsilon },L^{q-\varepsilon }\right) } &\leq &\varepsilon
^{\frac{\theta }{p-\varepsilon }}\left\Vert f_{n_{0}}-f\right\Vert
_{_{W\left( L^{p-\varepsilon },L^{q-\varepsilon }\right) }}+\varepsilon ^{%
\frac{\theta }{p-\varepsilon }}\left\Vert f_{n_{0}}\right\Vert _{_{W\left(
L^{p-\varepsilon },L^{q-\varepsilon }\right) }} \\
&\leq &\left\Vert f_{n_{0}}-f\right\Vert _{W\left( L^{p),\theta
},L^{q),\theta }\right) }+\frac{\eta }{2}<\frac{\eta }{2}+\frac{\eta }{2}%
=\eta
\end{eqnarray*}%
when $\varepsilon <\varepsilon _{0}.$ This completes the proof.
\end{proof}

\begin{proposition}
If $q\leq p,$ and $\theta \geq 1,$ then the set $C_{0}^{\infty }\left(
\Omega \right) $ is not dense in $W\left( L^{p),\theta },L^{q),\theta
}\right) \left( \Omega \right) .$
\end{proposition}

\begin{proof}
It is enough to give the proof for $\ $the special case $\Omega =\left(
0,1\right) .$ Let $f\left( t\right) =t^{\frac{1}{p}},$ for $p>1.$ We showed
in Proposition $2,$ that $t^{-\frac{1}{p}}\in W\left( L^{p),\theta
},L^{q),\theta }\right) \left( 0,1\right) .$ In the other hand since $%
\overline{C_{0}^{\infty }}\mid _{W\left( L^{p),\theta },L^{q),\theta
}\right) }\subset $ $\overline{C_{0}^{\infty }}\mid _{L^{p),\theta }}$ and $%
t^{-\frac{1}{p}}\notin \overline{C_{0}^{\infty }}\left( 0,1\right) \mid
_{L^{p),\theta }},$ then $t^{-\frac{1}{p}}\notin $ $\overline{C_{0}^{\infty }%
}\mid _{W\left( L^{p),\theta },L^{q),\theta }\right) }.$ Thus $t^{-\frac{1}{p%
}}\in $ $W\left( L^{p),\theta },L^{q),\theta }\right) \smallsetminus 
\overline{C_{0}^{\infty }}\mid _{W\left( L^{p),\theta },L^{q),\theta
}\right) }.$ This ends the proof.
\end{proof}

\begin{definition}
Let $\left( X,\left\Vert .\right\Vert _{X}\right) $ be a Banach function
space and let $f\in X$ be any arbitrary function. We say that $f$ has
absolutely continuous norm in $X$ if%
\begin{equation*}
\lim_{n\rightarrow \infty }\left\Vert f\chi _{E_{n}}\right\Vert _{X}=0
\end{equation*}%
for every sequence $\left\{ E_{n}\right\} _{n\in 
%TCIMACRO{\U{2115} }%
%BeginExpansion
\mathbb{N}
%EndExpansion
}$ satisfying $E_{n}\rightarrow \phi .$\textit{We will denote the subspace
of the functions in }$X$\textit{\ with absolutely continuous norm by }$X_{a}$%
\textit{. If }$X=X_{a}$\textit{, then }$X$\textit{\ it self is said to have
absolute continuous norm}$.$
\end{definition}

We need the following known two theorems to find the dual of the grand
Wiener amalgam spaces..

\begin{theorem}
( See $\left[ 3\right] ,$ Corollary $4.3$). The dual space $X^{\ast }$of a
Banach function space $X$ is isometric to the associate space $X^{^{\prime
}} $ if and only if $X$ has absolutely continuous norm.
\end{theorem}

\begin{theorem}
( See $\left[ 3\right] ,$ Corollary $4.4$). A Banach function space is
reflexive if and only if both $X$ and its associate space $X^{^{\prime }}$
have absolute continuous norm.
\end{theorem}

\begin{theorem}
The grand Wiener amalgam spaces $W\left( L^{p),\theta },L^{q),\theta
}\right) \left( c,d\right) $ has not absolute continuous norm.
\end{theorem}

\begin{proof}
We will give the proof of this theorem for the interval $\left( c,d\right)
=\left( 0,1\right) .$ One can prove similarly this theorem for any interval $%
\left( c,d\right) .$ Take the function $f\left( t\right) =t^{-\frac{1}{p}}.$
We will show that $f$ is not absolute continuous. Let $Q=\left( 0,a\right)
\subset \left( 0,1\right) .$ Then%
\begin{equation}
\left\Vert \left( t^{-\frac{1}{p}}\chi _{Q}\right) \chi _{Q+x}\right\Vert
_{p),\theta }=\left\Vert t^{-\frac{1}{p}}\chi _{Q\cap Q+x}\right\Vert
_{p),\theta }=\left\Vert t^{-\frac{1}{p}}\chi _{\left( x,a\right)
}\right\Vert _{p),\theta }  \tag{19}
\end{equation}%
\begin{eqnarray}
&=&\sup_{0<\varepsilon \leq p-1}\varepsilon ^{\frac{\theta }{p-\varepsilon }%
}\left\Vert t^{-\frac{1}{p}}\chi _{\left( x,a\right) }\right\Vert
_{p-\varepsilon }=\sup_{0<\varepsilon \leq p-1}\varepsilon ^{\frac{\theta }{%
p-\varepsilon }}\left( \int_{x}^{a}\left\vert t^{-\frac{1}{p}}\right\vert
^{p-\varepsilon }\right) ^{\frac{1}{p-\varepsilon }}  \notag \\
&=&\sup_{0<\varepsilon \leq p-1}\varepsilon ^{\frac{\theta }{p-\varepsilon }%
}\left( \int_{x}^{a}t^{-\frac{p-\varepsilon }{p}}\right) ^{\frac{1}{%
p-\varepsilon }}=\sup_{0<\varepsilon \leq p-1}\varepsilon ^{\frac{\theta }{%
p-\varepsilon }}\left( \int_{x}^{a}t^{-1+\frac{\varepsilon }{p}}\right) ^{%
\frac{1}{p-\varepsilon }}  \notag \\
&=&\sup_{0<\varepsilon \leq p-1}\varepsilon ^{\frac{\theta }{p-\varepsilon }%
}\left( \frac{p}{\varepsilon }\left[ a^{\frac{\varepsilon }{p}}-x^{\frac{%
\varepsilon }{p}}\right] \right) ^{\frac{1}{p-\varepsilon }%
}=\sup_{0<\varepsilon \leq p-1}\varepsilon ^{\frac{\theta -1}{p-\varepsilon }%
}p^{\frac{1}{p-\varepsilon }}\left[ a^{\frac{\varepsilon }{p}}-x^{\frac{%
\varepsilon }{p}}\right] ^{\frac{1}{p-\varepsilon }}.  \notag
\end{eqnarray}%
It is easy to see that $\varepsilon ^{\frac{\theta -1}{p-\varepsilon }}$ and 
$p^{\frac{1}{p-\varepsilon }}$ are increasing\ functions of $\varepsilon .$
Since $0<a<1,$ from $\left( 19\right) $ we have%
\begin{equation}
\left\Vert \left( t^{-\frac{1}{p}}\chi _{Q}\right) \chi _{Q+x}\right\Vert
_{p),\theta }=\left( p-1\right) ^{\theta -1}p\left[ 1-x^{\frac{p-1}{p}}%
\right] .  \tag{20}
\end{equation}%
Also since $x$ tends to $0$ as $a\rightarrow 0$, from $\left( 20\right) $ we
obtain%
\begin{eqnarray*}
\lim_{a\rightarrow 0}\left\Vert t^{-\frac{1}{p}}\chi _{\left( 0,a\right)
}\right\Vert _{W\left( L^{p),\theta },L^{q),\theta }\right) }
&=&\lim_{a\rightarrow 0}\left\Vert \left\Vert t^{-\frac{1}{p}}\chi _{\left(
x,a\right) }\right\Vert _{p),\theta }\right\Vert _{q),\theta } \\
&=&\lim_{a\rightarrow 0}\left\Vert \left( p-1\right) ^{\theta -1}p\left[
1-x^{\frac{p-1}{p}}\right] \right\Vert _{q),\theta } \\
&=&\left\Vert \lim_{a\rightarrow 0}\left( p-1\right) ^{\theta -1}p\left[
1-x^{\frac{p-1}{p}}\right] \right\Vert _{q),\theta } \\
&=&\left( p-1\right) ^{\theta -1}p\left\Vert \lim_{x\rightarrow 0}\left(
1-x^{\frac{p-1}{p}}\right) \right\Vert _{q),\theta } \\
&=&\left( p-1\right) ^{\theta -1}p\left\Vert 1\right\Vert _{q),\theta }\neq
0.
\end{eqnarray*}%
Thus $f\left( t\right) =t^{-\frac{1}{p}}\in W\left( L^{p),\theta
},L^{q),\theta }\right) \left( 0,1\right) ,$ but it has not absolute
continuous norm in $W\left( L^{p),\theta },L^{q),\theta }\right) \left(
0,1\right) .$ Then by definition $3,$ $W\left( L^{p),\theta },L^{q),\theta
}\right) \left( 0,1\right) $ has not absolute continuous norm.
\end{proof}

\begin{corollary}
The grand Wiener amalgam spaces $W\left( L^{p),\theta },L^{q),\theta
}\right) \left( c,d\right) $ is not reflexive.
\end{corollary}

\begin{proof}
The proof is clear from Theorem 9 and Theorem 10.
\end{proof}

\begin{corollary}
The dual space $\left( W\left( L^{p),\theta },L^{q),\theta }\right) \left(
c,d\right) \right) ^{\ast }$of the grand Wiener amalgam space $W\left(
L^{p),\theta },L^{q),\theta }\right) \left( c,d\right) $ is not isometric to
the associate space $\left( W\left( L^{p),\theta },L^{q),\theta }\right)
\left( c,d\right) \right) ^{^{\prime }}$of this grand Wiener amalgam space $%
W\left( L^{p),\theta },L^{q),\theta }\right) \left( c,d\right) $.
\end{corollary}

\begin{proof}
The proof of this theorem is easy from Theorem 8 and Theorem 10.
\end{proof}

\begin{theorem}
Let $\frac{p}{p-1}=p^{^{\prime }},\frac{q}{q-1}=q^{^{\prime }}$ and $%
1<p,q<\infty .$ Then 
\begin{equation*}
\left( W\left( L^{p),\theta },L^{q),\theta }\right) \left( \Omega \right)
\right) ^{^{\prime }}=W\left( L^{p)^{\prime },\theta },L^{q)^{\prime
},\theta }\right) \left( \Omega \right) ,
\end{equation*}%
and the norms $\left\Vert g\right\Vert _{W\left( L^{p)^{^{\prime
}}},L^{q)^{^{\prime }}}\right) }$ and%
\begin{equation*}
\left\Vert g\right\Vert _{W\left( L^{p),\theta },L^{q),\theta }\right)
^{^{\prime }}}=\left\{ 
\begin{array}{c}
\sup \dint\limits_{\Omega }\left\vert f\left( x\right) g\left( x\right)
\right\vert dx:f\in W\left( L^{p),\theta },L^{q),\theta }\right) \left(
\Omega \right) ,\text{ } \\ 
\left\Vert f\right\Vert _{W\left( L^{p)},L^{q)}\right) }\leq 1%
\end{array}%
\right\} 
\end{equation*}%
are equivalent.
\end{theorem}

\begin{proof}
If $f\in W\left( L^{p),\theta },L^{q),\theta }\right) \left( \Omega \right) $
and $g\in W\left( L^{p)^{\prime },\theta },L^{q)^{\prime },\theta }\right)
\left( \Omega \right) ,$ then%
\begin{equation*}
\left\Vert f\right\Vert _{W\left( L^{p),\theta },L^{q),\theta }\right)
}.\left\Vert g\right\Vert _{W\left( L^{p)^{\prime },\theta },L^{q)^{\prime
},\theta }\right) }=\left\Vert \left\Vert f.\chi _{Q+x}\right\Vert
_{p),\theta }\right\Vert _{q),\theta }\left\Vert \left\Vert g.\chi
_{Q+x}\right\Vert _{p)^{^{\prime }},\theta }\right\Vert _{q)^{^{\prime
}},\theta }.
\end{equation*}%
Since $\left\Vert f.\chi _{Q+x}\right\Vert _{p),\theta }\in L^{q),\theta }$
and $\left\Vert g.\chi _{Q+x}\right\Vert _{L^{p)^{\prime },\theta }}\in
L^{q)^{\prime },\theta },$ by the H\"{o}lder inequality for the generalized
grand Lebesgue space $\left\Vert f.\chi _{Q+x}\right\Vert _{p),\theta
}\left\Vert g.\chi _{Q+x}\right\Vert _{L^{p)^{\prime },\theta }}\in
L^{1}\left( \Omega \right) $ and%
\begin{equation}
\left\Vert \left\Vert f.\chi _{Q+x}\right\Vert _{p),\theta }\left\Vert
g.\chi _{Q+x}\right\Vert _{p)^{^{\prime }},\theta }\right\Vert _{1}\leq
\left\Vert \left\Vert f.\chi _{Q+x}\right\Vert _{p),\theta }\right\Vert
_{q),\theta }\left\Vert \left\Vert g.\chi _{Q+x}\right\Vert _{p)^{^{\prime
}},\theta }\right\Vert _{q)^{^{\prime }},\theta }.  \tag{21}
\end{equation}%
Also since $f.\chi _{Q+x}\in L^{p),\theta }$ and $g.\chi _{Q+x}\in
L^{p)^{\prime },\theta },$ one more applying the H\"{o}lder inequality for
the generalized Lebesgue space we have%
\begin{equation}
\left\Vert fg\right\Vert _{1}=\left\Vert fg\right\Vert _{W\left(
L^{1},L^{1}\right) }=\left\Vert \left\Vert \left( f.\chi _{Q+x}\right)
\left( g.\chi _{Q+x}\right) \right\Vert _{1}\right\Vert _{1}  \tag{22}
\end{equation}%
\begin{equation}
\leq \left\Vert \left\Vert f.\chi _{Q+x}\right\Vert _{p),\theta }\left\Vert
g.\chi _{Q+x}\right\Vert _{L^{p)^{\prime },\theta }}\right\Vert _{1}.  \notag
\end{equation}%
Combining $\left( 21\right) $ and $\left( 22\right) $ we obtain \ 
\begin{equation*}
\dint\limits_{\Omega }\left\vert f\left( x\right) .g\left( x\right)
dx\right\vert =\left\Vert f.g\right\Vert _{1}\leq \left\Vert f\right\Vert
_{W\left( L^{p)},L^{q)}\right) }.\left\Vert g\right\Vert _{W\left(
L^{p)^{^{\prime }}},L^{q)^{^{\prime }}}\right) }.
\end{equation*}%
From this inequality we have%
\begin{equation*}
\left\Vert g\right\Vert _{W\left( L^{p),\theta },L^{q),\theta }\right)
^{^{\prime }}}=\left\{ 
\begin{array}{c}
\sup \dint\limits_{\Omega }\left\vert f\left( x\right) g\left( x\right)
\right\vert dx:f\in W\left( L^{p),\theta },L^{q),\theta }\right) \left(
\Omega \right) , \\ 
\text{ }\left\Vert f\right\Vert _{W\left( L^{p)},L^{q)}\right) }\leq 1%
\end{array}%
\right\} \leq \left\Vert g\right\Vert _{W\left( L^{p)^{^{\prime
}}},L^{q)^{^{\prime }}}\right) }.
\end{equation*}%
This implies%
\begin{equation*}
W\left( L^{p)^{\prime },\theta },L^{q)^{\prime },\theta }\right) \left(
\Omega \right) \subset \left( W\left( L^{p),\theta },L^{q),\theta }\right)
\left( \Omega \right) \right) ^{^{\prime }}
\end{equation*}%
If one uses the same technic in $\left( \left[ 13\right] ,\text{Theorem }%
11.7.1\left( c\right) \right) ,$ obtains 
\begin{equation*}
W\left( L^{p)^{\prime },\theta },L^{q)^{\prime },\theta }\right) \left(
\Omega \right) \supset \left( W\left( L^{p),\theta },L^{q),\theta }\right)
\left( \Omega \right) \right) ^{^{\prime }}.
\end{equation*}%
Thus we have 
\begin{equation*}
W\left( L^{p)^{\prime },\theta },L^{q)^{\prime },\theta }\right) \left(
\Omega \right) =\left( W\left( L^{p),\theta },L^{q),\theta }\right) \left(
\Omega \right) \right) ^{^{\prime }}.
\end{equation*}%
Since the spaces $W\left( L^{p)^{\prime },\theta },L^{q)^{\prime },\theta
}\right) \left( \Omega \right) $ and $\left( W\left( L^{p),\theta
},L^{q),\theta }\right) \left( \Omega \right) \right) ^{^{\prime }}$ are
equal and they are Banach spaces with respect to the norms $\left\Vert
.\right\Vert _{W\left( L^{p)^{^{\prime }}},L^{q)^{^{\prime }}}\right) }$ and 
\begin{equation*}
\left\Vert g\right\Vert _{W\left( L^{p),\theta },L^{q),\theta }\right)
^{^{\prime }}}=\left\{ \sup \dint\limits_{\Omega }\left\vert f\left(
x\right) g\left( x\right) \right\vert dx:f\in W\left( L^{p),\theta
},L^{q),\theta }\right) \left( \Omega \right) ,\text{ }\left\Vert
f\right\Vert _{W\left( L^{p)},L^{q)}\right) }\leq 1\right\} ,
\end{equation*}%
respectively, then these norms are equivalent by Two \ norm theorem (see
Theorem 7.3.3, $\left[ 19\right] )$
\end{proof}

\end{document}